\documentclass[12pt,dvips]{amsart}
 
\usepackage{amsmath}
\usepackage{bbm}
\usepackage[OT1]{fontenc}
\usepackage{txfonts}
\usepackage{color}
\usepackage{epic}
\usepackage{eepic}
\usepackage{graphicx}
\usepackage{paralist}
\usepackage[tight,hang,TABTOPCAP]{subfigure}
\usepackage[small,it]{caption2}
\usepackage{tabularx,booktabs}
\usepackage[mathscr]{eucal}
\usepackage{units}
\usepackage{tikz}
\usepackage{upgreek}

\usepackage{url}
\usepackage{color}
\definecolor{darkblue}{rgb}{.8,.15,.15}
\definecolor{darkgreen}{rgb}{0.15,.4,.5}
\usepackage[pagebackref,
            naturalnames,
            colorlinks,linkcolor=darkblue,
            citecolor=darkblue,pagecolor=darkblue,
            filecolor=darkblue,urlcolor=darkgreen,
            anchorcolor=darkblue,menucolor=darkblue,
            bookmarksopen=false,bookmarks=false,
            ps2pdf=true
           ]{hyperref}

\renewcommand*{\backref}[1]{}
\renewcommand*{\backrefalt}[4]{%
\ifcase #1 %
 [Not cited]%
\or
 [Cited on page #2]%
\else
 [Cited on pages #2]%
\fi
}

\newtheorem{theorem}{Theorem}[section]
\newtheorem{corollary}[theorem]{Corollary}
\newtheorem{lemma}[theorem]{Lemma}

\theoremstyle{definition}

\numberwithin{theorem}{section}
\numberwithin{figure}{section}
\numberwithin{equation}{section}
\numberwithin{table}{section}

\newcommand{\x}{\boldsymbol{x}}
\renewcommand{\v}{\boldsymbol{v}}

\newcommand{\R}{\mathbb{R}}

\newcommand{\TT}{\mathbb{T}}
\renewcommand{\P}{\mathbb{P}}

\title[Non-Generic Tropical Hyperplane Arrangements ]{\mbox{}\vspace{-3\baselineskip}
  \\
  Non-Generic Tropical Hyperplane Arrangements and The Secondary Polytope of $\Delta_{n-1} \times \Delta_{d-1}$}

\author[Piechnik]{Lindsay C. Piechnik}
\address{Lindsay C. Piechnik \\ Columbia University, NY \\ USA}
\email{piechnik@math.columbia.edu}

\date{\today}
\begin{document}
\begin{abstract}

Ardila and Develin's paper on tropical oriented hyperplane arrangements and tropical oriented matroids  defines tropical oriented matroids and conjectures a bijection between them and triangulations of products of simplices $\Delta_{n-1} \times \Delta_{d-1}$.  Oh and Yoo recently confirmed this conjecture; however, neither group addressed the case of hyperplanes that are not in generic position.  These non-generic arrangements do not correspond to tropical oriented matroids, but they encode information about subdivisions of $\Delta_{n-1} \times \Delta_{d-1}$. This note considers the non-generic case and presents some preliminary results in the area.  

\vspace{-3\baselineskip}
\end{abstract}
\maketitle

{\tiny
\tableofcontents}


\section{Tropical Hyperplanes Background}

 Working with tropical hyperplane arrangements requires some familiarity with tropical geometry.   
   Tropical geometry is done over the \emph{tropical semiring}.  
   The \emph{tropical semiring} is the ring $R = \{ \R, \otimes, \oplus \}$, where $\otimes$ is standard addition and $\oplus$ is taking maximums.  
    In the study of tropical hyperplane arrangements, one benefits from considering \emph{tropical projective (d-1) space}  $\TT\P^{d-1}$. 
        This is obtained from Tropical $\R^d$ by modding out by tropical scalar multiplication.  
        (For further background see \cite{MR2149011}.)

     A tropical hyperplane is given by the vanishing of a linear functional $\sum{c_ix_i}$.
    Unlike traditional hyperplanes whose half-spaces can be labeled by sign $(0, -, +)$, a tropical hyperplane divides  $\TT\P^{d-1}$ into $d$ full dimensional sectors.  
    This creates a fan polar to the standard simplex $\Delta_{d-1}$.
      There is a natural labeling of both the full dimensional sectors and the lower dimensional intersections of the fan that is the tropical hyperplane. 
        The cones of the fan are indexed by the subset of $[d]$ corresponding to the $c_i + x_i$ that it maximizes. 
         For  full dimensional sectors this is a single element, as the maximum is only achieved once.
           Each lower dimensional cone is indexed by  basis  vectors of the faces of the simplex to which it is polar. 
            For the apex this means the entire set $[d]$.           
            Positions within a tropical hyperplane arrangement are given by ordered tuples of these $[d]$-subsets.

  Points in a specific tropical hyperplane arrangement can be described by their projective coordinates.  However, when interested in the purely combinatorial properties of such an arrangement, only the relative positions of the hyperplanes matter.  So, points are typically described by their types.  
  The \emph{type} of a point $\x$ $\in$ $\TT\P^{d-1}$ with respect to a tropical hyperplane arrangement $H_1, ... , H_n$ $\in$ $\TT\P^{d-1}$ is the $n$-tuple $(A_1, ... A_n)$, where each $A_i$ is the subset of $[d]$ corresponding to the closed  sectors of the hyperplane $H_i$ in which $\x$ is contained.  With respect to algebraic coordinates, for the hyperplane $H_i$ with vertex $\v_i$ $=$ $(v_{i1}, ... v_{id})$,  $A_i$ indicates which among the $x_j - v_{ij}$ are maximized.
    It is clear that every point on a face of a given arrangement has the same type.   So these types do not distinguish points within a face, but they do encode all relative information about the faces of the arrangements.      And the collection of types of an arrangement determine the arrangement up to combinatorial equivalence \cite{MR2511751}.


\section{Tropical Oriented Matroid Axioms}

Ardila and Develin's definition of a tropical oriented matroid was inspired by both  tropical oriented hyperplanes and traditional matroid theory \cite{MR2511751}.  As such they aimed to establish many parallel properties to the relationship between traditional oriented hyperplanes and oriented matroid theory. 

Tropical oriented matroids are defined via axioms on types. 
 
 A \emph{tropical oriented matroid} $M$ with parameters $(n, d)$ is defined as a collection of $(n, d)$ types which satisfy the Boundary, Elimination, Comparablity, and Surrounding Axioms.  
 
 \emph{Boundary Axiom:}  for each $j$ in $[d]$,  $\j ;= (j, j, ... , j)$ is a type in $M$.
 
 \emph{Elimination Axiom:} For any two types $A$ and $B$ in $M$ and a position $j \in [n]$, there is a type $C$ in $M$ such that $C_j = A_j \cup B_j$ , and $C_k \in \{ A_k, B_k, A_k \cup B_k \}$ for each $k \in [n]$.
 
 \emph{Comparability Axiom:} For any two types $A$ and $B$ in $M$ the comparability graph $CG_{A,B}$ is acyclic.
 
 \emph{Surrounding Axiom:} If $A$ is a type of $M$ then every refinement of $A$ is also a type of $M$.
 
 Ardila and Develin's paper \cite{MR2511751} gives complete explanations of these axioms, background definitions and interpretations.
 But the spirit of these axioms is natural even if their wording isn't.   
     For example, the surrounding axiom, which is most pertinent to this paper, says that given a type $A$, for which some entries are not singletons, a refinement of $A$ can be obtained by moving infinitesimally away from the face of the fan which $A$ represents. 
      In particular if both $j$ and $k$ appear in the $i^{th}$ coordinate of $A$'s type then a refinement of $A$ can be obtained by moving infinitesimally in the $j$ direction, or the $k$ direction, thus breaking the tie, $x_j - v_{ij} = x_k - v_{ik}$.  
     
      The failure of this axiom is what prevents the non-generic tropical hyperplane arrangements from being tropical oriented matroids.


\section{Tropical Hyperplane Arrangements and Tropical Matroids}

Traditional hyperplane arrangements in $d$ space intersect either   in a $d-2$-dimensional linear space (as the planes themselves are $d-1$- dimensional linear spaces of $d$- space), or they are parallel.  
This is not the case in tropical arrangements. 
 All tropical arrangements intersect.  
 However, they do not all intersect the same way. 
  Generally the intersection of two hyperplanes in $\TT \P^{d-1}$ is a $d-3$ dimensional cone. 
    However, this is not true when an apex of one of the tropical hyperplanes, say $H_i$, occurs on proper face of the fan determined by one of the other hyperplanes, say $H_j$, in the arrangement. 
     The result is that some of the cones in the fan determined by $H_i$ will be proper subsets of $H_j$'s. 
 Arrangements in which this occurs are called \emph{non-generic}.
   Arrangements that aren't non-generic are called \emph{generic}.
   The behavior of non-generic topical hyperplane arrangements differs from that of generic hyperplane arrangements.  
A \emph{non-generic} apex $A$ of a tropical hyperplane arrangement is an apex of a tropical hyperplane arrangement which is located on a proper face of the fan given by one of the other hyperplanes of the arrangement.  Figure \ref{fig:visual} shows two  arrangements of three tropical hyperplanes in $\TT \P^{2}$, one generic and one non-generic.

     \begin{figure}[htb]
  \centering
\includegraphics[width=30mm]{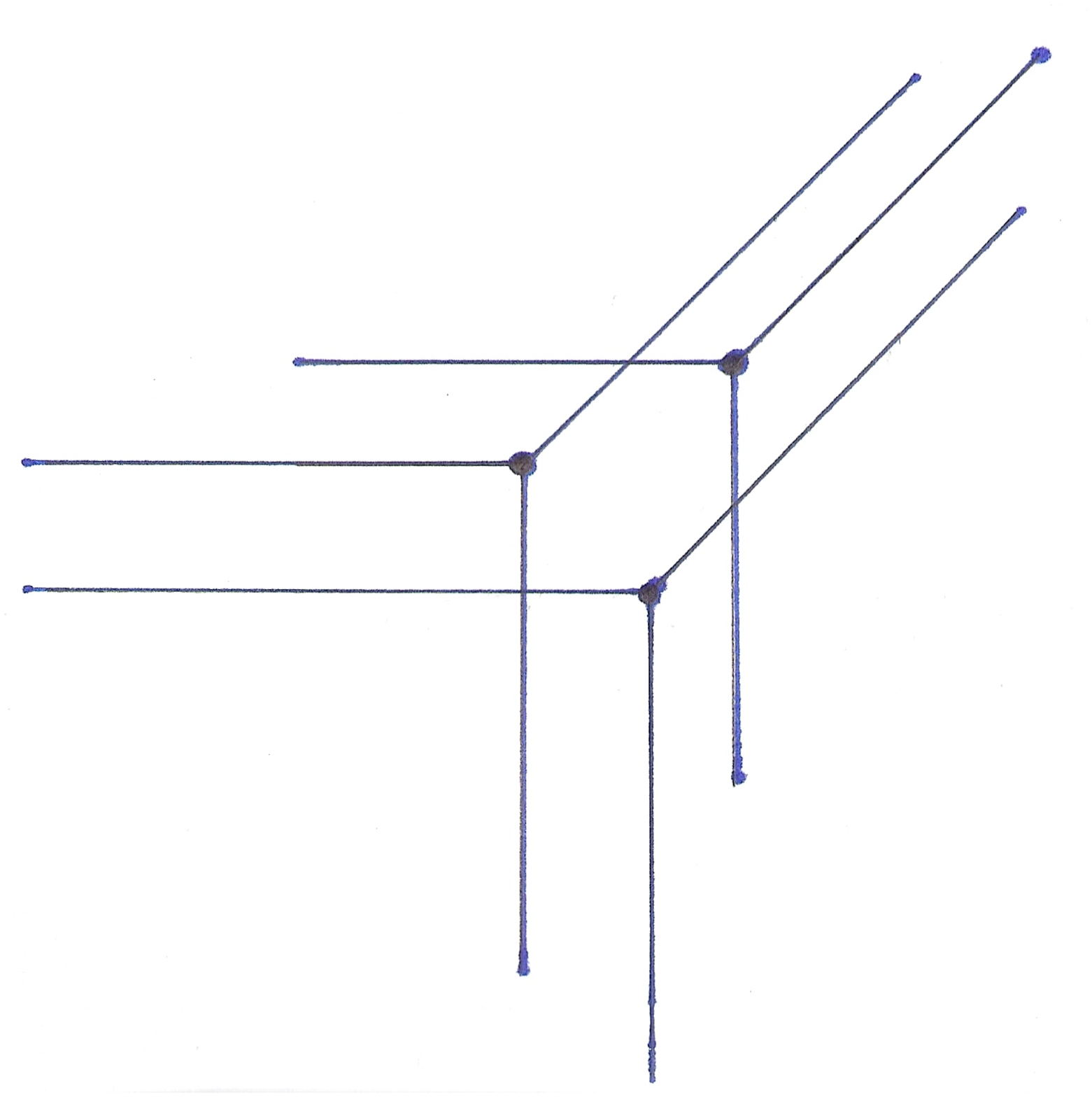}
\includegraphics[width=30mm]{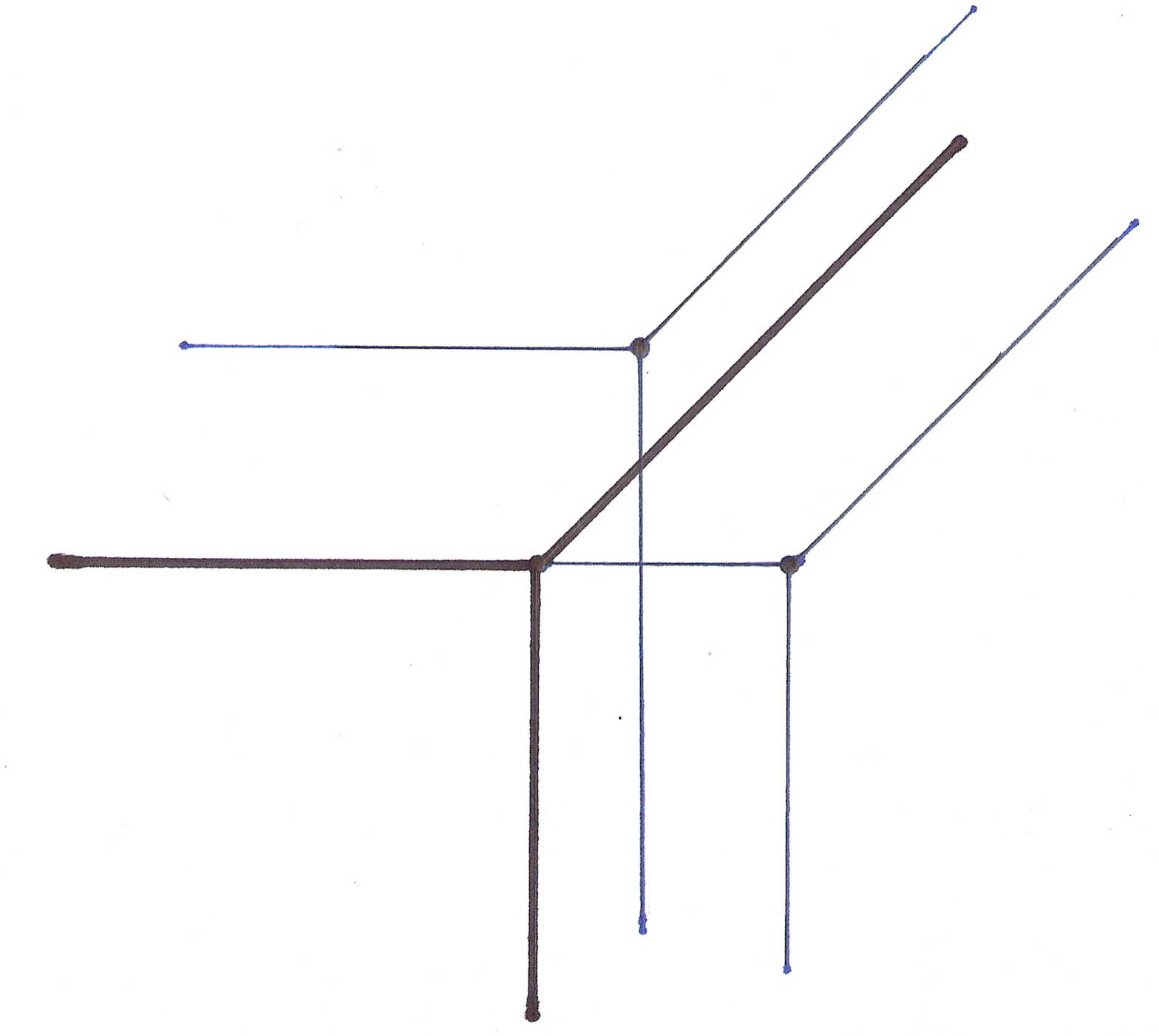}
  \caption{
  Here are two arrangements of three tropical hyperplanes in $\TT \P^{2}$.  The one on the left is generic, the one on the right is non-generic.  (The non-generic hyperplane is bold.)
  }
  \label{fig:visual}
\end{figure}

\begin{lemma} \label{lem:nonGenVert}

   For a non-generic apex $A$, in an arrangement of $n$ tropical hyperplanes in $\TT \P^{d-1}$ , 
 the total number of elements summed over its $A_i$'s will be strictly greater than  $n + d-1$.

\end{lemma}

\begin{proof}

  Let $M$ be an arrangement of $n$ tropical hyperplanes in $\TT \P^{d-1}$  with a non-generic hyperplane.  Without loss of generality assume $H_1$ is this non-generic hyperplane.  By  definition, its non-generic apex $A$ occurs on a proper subface of the polyhedral fan given by one of the other hyperplanes, say $H_i$ of $M$.    Which means the coordinates (actual coordinates, not type) of $A$ must satisfy equality on at least two of the 
$\sum{c_jx_j} = max \{ c_1 + x_1, ... , c_d + x_d\}$ defining $H_i$, say this includes $j$ and $k$, then $A_i$ contains at least $j$ and $k$.  Since $A$ is $H_1$'s apex, we know  $A_1 = [d]$.  That gives us at least $d + 2$ entries from 2 positions in $A$'s type.  We know there are $n - 2$ positions remaining and that none of them are empty.  So $A$'s type contains at least $n + d$ elements.

\end{proof}

\begin{corollary} \label{cor:NonGenericHyps}

   The types of a non-generic tropical hyperplane arrangement do not form a tropical oriented matroid.
 
\end{corollary}

\begin{proof}

 Let $M$ be a non-generic tropical hyperplane arrangement.  By definition $M$ has a non-generic apex $A$.  Consider this $A$, with projective coordinates $\{v_{11}, ... v_{1d}\}$.  Without loss of generality assume $A$ is the apex of $H_1$.  By lemma \ref{lem:nonGenVert}, the sum of elements over the sets $A_i$ is at least $n + d $.  Without loss of generality assume it has $n + d $ elements.   Since $A$ is the apex of $H_1$ we know $A_1 = [d]$ so there are $n-1$ sets for the remaining $n$ elements.  None of these sets can be empty.  So there must be one with $2$ elements, say $A_i = \{ j, k \}$.  
 So, $x_j - v_{ij} = x_k - v_{ik}$, and $x_1 - v_{11} = x_2 - v_{12}  = .....  =  x_d - v_{1d}$.
 Now we can move infinitesimally off of this vertex in the direction to break the tie $x_j - v_{ij} = x_k - v_{ik}$, and obtain a point $A'$ which agrees on all coordinates of $A$ except $A_i$, where we will have  $A'_i = \{k\}$.  But this will also break the tie $x_j - v_{1j} = x_k - v_{1k}$.  So we will have $A'_1 = \{1, 2, ... j-1, j+1, ... d\} \ne A_1$.  So there is no point in the arrangement whose type agrees with $A$ on all coordinates except $A_i$ and has $i^{th}$ coordinate $\{k\}$.  This fails to satisfy the surrounding axiom for tropical oriented matroids.  So $M$ is not a tropical oriented matroid.

\end{proof}

In light of the previous corollary, Ardila and Develin's statement of the relationship between tropical oriented matroids and tropical hyperplane arrangements should be restated as follows:

\begin{theorem} \label{Thm:GenVert}

  The collection of types in a generic tropical hyperplane arrangement forms a tropical oriented matroid.

\end{theorem}

  Theorem 3.6 of
   \cite{MR2511751}
  says this without explicitly stating the generic assumption.  The proof of Theorem \ref{Thm:GenVert} is precisely their proof with an added generic qualification for assertions about the surrounding axiom.


\section{Tropical Hyperplane Arrangements Meet $\Delta_{n-1} \times \Delta_{d-1}$}

    Duality is a useful property for any matroid theory to have, but it is not  so obvious in the tropical oriented case.  In fact, it appeared as a conjecture in  Ardila and Develin's work. 
    However, it is true and a direct corollary of the affirmative answer to their primary conjecture, which was recently confirmed by Oh and Yoo.
    
    \begin{theorem}
    \label{them:Matroid/Simplex}
    
    A collection of $(n, d)$ -types is a tropical oriented matroid if and only if it corresponds to a triangulation of $\Delta_{n-1} \times \Delta_{d-1}$.
    
    \end{theorem}
    
    Oh and Yoo's proof uses, and is phrased with respect to, the bijection between products of the form $\Delta_{n-1} \times \Delta_{d-1}$ and the complete bipartide graph $k_{n,d}$ and appears in \cite {HoYoo}.    (For more on translations between $\Delta_{n-1} \times \Delta_{d-1}$ and the complete bipartide graph $k_{n,d}$ see \cite{SantosBook}.)

The bijection between tropical oriented matroids of type $(n, d)$ and products of simplices $\Delta_{n-1} \times \Delta_{d-1}$ has ties to non-generic arrangements of tropical hyperplanes as well. 
  
    In particular, the fact that tropical oriented matroids are in bijection with triangulations of products of simplices means that non-generic tropical hyperplane arrangements do not correspond to triangulations of products of simplices.  Such non-generic arrangements exist instead as the limit of two such triangulations--- two triangulations separated by a ``flip."  The non-generic nature of the arrangement corresponds to the that the flip remains unspecified.       
    When all triangulations of $\Delta_{n-1} \times \Delta_{d-1}$ are regular, all type $(n, d)$ tropical oriented matroids correspond to vertices of $\Delta_{n-1} \times \Delta_{d-1}$ 's secondary polytope.

    (For nice background and explanation of triangulations, flips, and secondary polytopes see \cite{SantosBook}.)
    
    Combining this with the relationship between generic and non-generic tropical hyperplane arrangements, we obtain the following statement:

    \begin{theorem} \label{them:NonGenFlip}
    
      When all triangulations of $\Delta_{n-1} \times \Delta_{d-1}$ are regular, all non-generic arrangements of $n$ tropical hyperplane arrangements in $\TT\P^{d-1}$ correspond to faces of dim $> 0$ of the secondary polytope of $\Delta_{n-1} \times \Delta_{d-1}$.
      
      \end{theorem}

      \begin{proof}
      
      Regular triangulations of $\Delta_{n-1} \times \Delta_{d-1}$  correspond to vertices of the secondary polytope of $\Delta_{n-1} \times \Delta_{d-1}$.  
      When all triangulations of $\Delta_{n-1} \times \Delta_{d-1}$ are regular, this means that arrangements of $n$ tropical hyperplanes in $\TT\P^{d-1}$ also correspond to vertices of the secondary polytope.
      Edges between vertices in this polytope correspond to flips  between triangulations.  In particular, a subdivision that does not make a choice of flip is only a subdivision, not a triangulation, corresponding to an edge (or higher dimensional face of the polytope if it fails to make a choice on multiple flips).  
      A generic tropical hyperplane arrangement $M$  
       corresponds to a tropical oriented matroid, and hence a vertex of the secondary polytope of $\Delta_{n-1} \times \Delta_{d-1}$. 
       I abuse notation by also labeling this $M$.   
       Consider an apex $A$ of this arrangement  changing relative position with respect to some other hyperplane $H_i$.  Without loss of generality,  say it moves from sector $k$ to sector $l$ (taking you from $A$ with $A_i = \{k\}$ to $A'$ with $A'_i = \{l\}$).  This move results in a new tropical oriented matroid, corresponding to a different vertex of the secondary polytope of $\Delta_{n-1} \times \Delta_{d-1}$, call it $M'$.  
       To move from sector $k$ to sector $l$ one must pass through a face of the fan given by $H_i$, in particular the face of $H_i$ labeled $\{k,l\}$.
        Pausing at that point would yield $\tilde{A}$ such that $\tilde{A}_i = \{k,l\}$.  
        The arrangement containing this $\tilde{A}$ is a non-generic tropical hyperplane arrangement, and corresponds to the face of the secondary polytope containing $M$ and $M'$.  
      
      \end{proof}

       \cite{SantosBook} explains that the only cases of  $\Delta_{n-1} \times \Delta_{d-1}$ for which $n$ and $d$ are $> 1$ and all triangulations are regular are: $\Delta_{2} \times \Delta_{2}$, $\Delta_{3} \times \Delta_{2}$, $\Delta_{4} \times \Delta_{2}$, $\Delta_{2} \times \Delta_{3}$, and $\Delta_{2} \times \Delta_{4}$.  
      The reason this result does not hold for other $\Delta_{n-1} \times \Delta_{d-1}$ is that non-regular triangulations are not represented by the secondary polytope.  
      It remains true that a non-generic tropical hyperplane arrangement corresponds to a subdivision  ``between" triangulations, which correspond to its neighboring generic hyperplane arrangements.  
      Yes, this statement suggests that there is a space of all possible tropical hyperplane arrangements of a given size and that the non-generic ones represent the boundaries separating the generic ones.
        This is the more general case that is being modeled here by the secondary polytope.

      These preliminary results demonstrate not only an important distinction between the generic and non-generic  tropical hyperplane arrangements but also that the non-generic case merits further investigation.

\bibliographystyle{plain}
\bibliography{treeing.}

\begin{thebibliography}{1}

\bibitem{MR2511751}
Federico Ardila and Mike Develin.
\newblock Tropical hyperplane arrangements and oriented matroids.
\newblock {\em Math. Z.}, 262(4):795--816, 2009.

\bibitem{SantosBook}
Jes\'us~A. De~Loera, Rambau J{\"o}rg, and Francisco Santos.
\newblock {\em Triangulations Structures for Algorithms and Applications}.
\newblock Algorithms and Computation in Mathematics, Vol. 25. Springer.,
  Boston, MA, 2010.

\bibitem{HoYoo}
Suho Oh and Hwanchul Yoo.
\newblock Triangulations of $\delta_{n-1} \times \delta_{d-1}$ and tropical
  oriented matroids.
\newblock 2010.

\bibitem{MR2149011}
J{\"u}rgen Richter-Gebert, Bernd Sturmfels, and Thorsten Theobald.
\newblock First steps in tropical geometry.
\newblock In {\em Idempotent mathematics and mathematical physics}, volume 377
  of {\em Contemp. Math.}, pages 289--317. Amer. Math. Soc., Providence, RI,
  2005.

\end{thebibliography}
  
\end{document}